\documentclass[12pt,reqno]{amsart}

\usepackage{amssymb}
\usepackage{amscd}

\usepackage[usenames,dvipsnames]{color}

\usepackage{euler,eucal} 

\usepackage[all,cmtip]{xy}

\usepackage{fancyhdr}
\theoremstyle{definition}

\newtheorem{claim}{}

\newcommand{\artlabel}[1]{{\sc #1.}}

\newif\ifdisplaynotes

\displaynotesfalse

\newcounter{piznote}

\newcounter{yknote}

\newcommand{\printname}[1]
       {\smash{\makebox[0pt]{\hspace{-1.0in}\raisebox{8pt}{\tiny #1}}}}
\newcommand{\labell}[1]
  {\ifdisplaynotes
      \label{#1}\printname{#1}
   \else
      \label{#1}
   \fi}

\newif\ifhighlighting

\highlightingtrue

\newcommand{\hide}[1]{}

\def \A{\ifmmode{{A}}\fi}
\def \B{\ifmmode{{B}}\fi}
\def \C{\ifmmode{{C}}\fi}
\def \D{\ifmmode{{D}}\fi}
\def \E{\ifmmode{{E}}\fi}
\def \F{\ifmmode{{F}}\fi}
\def \G{\ifmmode{{G}}\fi}
\def \H{\ifmmode{{H}}\fi}
\def \I{\ifmmode{{I}}\fi}
\def \J{\ifmmode{{J}}\fi}
\def \K{\ifmmode{{K}}\fi}
\def \L{\ifmmode{{L}}\fi}
\def \M{\ifmmode{{M}}\fi}
\def \N{\ifmmode{{N}}\fi}
\def \O{\ifmmode{{O}}\fi}
\def \P{\ifmmode{{P}}\fi}
\def \Q{\ifmmode{{Q}}\fi}
\def \R{\ifmmode{{R}}\fi}
\def \S{\ifmmode{{S}}\fi}
\def \T{\ifmmode{{T}}\fi}
\def \U{\ifmmode{{U}}\fi}
\def \V{\ifmmode{{V}}\fi}
\def \W{\ifmmode{{W}}\fi}
\def \X{\ifmmode{{X}}\fi}
\def \Y{\ifmmode{{Y}}\fi}
\def \Z{\ifmmode{{Z}}\fi}

\hide{
\def \A{\ifmmode{{\rm A}}\fi}
\def \B{\ifmmode{{\rm B}}\fi}
\def \C{\ifmmode{{\rm C}}\fi}
\def \D{\ifmmode{{\rm D}}\fi}
\def \E{\ifmmode{{\rm E}}\fi}
\def \F{\ifmmode{{\rm F}}\fi}
\def \G{\ifmmode{{\rm G}}\fi}
\def \H{\ifmmode{{\rm H}}\fi}
\def \I{\ifmmode{{\rm I}}\fi}
\def \J{\ifmmode{{\rm J}}\fi}
\def \K{\ifmmode{{\rm K}}\fi}
\def \L{\ifmmode{{\rm L}}\fi}
\def \M{\ifmmode{{\rm M}}\fi}
\def \N{\ifmmode{{\rm N}}\fi}
\def \O{\ifmmode{{\rm O}}\fi}
\def \P{\ifmmode{{\rm P}}\fi}
\def \Q{\ifmmode{{\rm Q}}\fi}
\def \R{\ifmmode{{\rm R}}\fi}
\def \S{\ifmmode{{\rm S}}\fi}
\def \T{\ifmmode{{\rm T}}\fi}
\def \U{\ifmmode{{\rm U}}\fi}
\def \V{\ifmmode{{\rm V}}\fi}
\def \W{\ifmmode{{\rm W}}\fi}
\def \X{\ifmmode{{\rm X}}\fi}
\def \Y{\ifmmode{{\rm Y}}\fi}
\def \Z{\ifmmode{{\rm Z}}\fi}
}

\newcommand{\NN}{{\bf N}}

\newcommand{\QQ}{{\bf Q}}
\newcommand{\RR}{{\bf R}}

\newcommand{\ZZ}{{\bf Z}}

\newcommand{\cG}{{\mathcal G}}

\newcommand{\cI}{{\mathcal I}}

\newcommand{\cO}{{\mathcal O}}

\newcommand{\cR}{{\mathcal R}}

\def \GL {\operatorname{GL}}
\def \Diff {\operatorname{Diff}}

\def \pr {\text{pr}}

\newcommand{\id}{{\bf 1}}

\newcommand{\orb}[1]{{#1}_\natural}
\newcommand{\orbbar}[1]{\overline{#1}_\natural}

\def \ol {\overline}

\begin{document}
	
  \title{Smooth Lie groups actions are parametrized diffeological subgroups}
  
  \author{Patrick Iglesias-Zemmour}
  
  \email{piz@math.huji.ac.il}
  
  \address{LATP/CNRS, Universit\'e de Provence, Marseille, France 
  }
  
  \author{Yael Karshon}
  
  \email{karshon@math.toronto.edu}
  
  \address{Department of Mathematics, The University of Toronto, 
   40 St.~George Street, Toronto, Ontario M5S 2E4, Canada.}
  
  
  
  \begin{abstract}
    We show that every effective smooth action of a
    Lie group $\G$ on a manifold $\M$ is a diffeomorphism from
    $\G$ onto its image in $\Diff(\M)$, where the image is 
    equipped with the subset diffeology of the functional
    diffeology.
  \end{abstract}
  
  \maketitle
    
  \section*{Introduction}
  
  The group $\Diff(\M)$ of diffeomorphisms of a 
  manifold $\M$, and subgroups that preserve certain geometric
  structures (for example, the group of symplectomorphisms of a
  symplectic manifold), are important objects of research.
  Such groups are often considered heuristically 
  as infinite dimensional Lie groups,
  whose ``finite dimensional Lie subgroups" are taken to be
  the images of smooth finite dimensional
  Lie group actions. 
  
  There are several rigorous approaches to the structure of
  $\Diff(\M)$.  When $\M$ is compact, $\Diff(\M)$ is a Fr\'echet Lie group
  \cite{Omo74}. 
  An approach through Fr\"olicher structures is given in \cite{KrMi97}.
  
  In this paper we take a simpler rigorous approach: 
  the diffeological approach.
  Roughly, a diffeology on a set $X$ declares which maps from $\RR^n$
  to $X$ are smooth, for all $n \in \NN$.
  See Section~\ref{sec:The-diffeological-framework} for relevant definitions.
  If $\M$ is a manifold, there is a natural diffeology on $\Diff(\M)$,
  with respect to which
  the notion ``Lie subgroup of $\Diff(\M)$'' is unambiguous:
  it is a subgroup that, equipped with the subset diffeology,
  is a manifold.
  The usual notion of an effective smooth action of a Lie group $\G$ 
  on a manifold $\M$ is the same thing as a group monomorphism 
  $\G \to \Diff(\M)$ that is smooth in the diffeological sense.    
  {\em A priori} it is not obvious that in this situation
  the manifold structure on $\G$ is induced from the ambient diffeology
  of $\Diff(\M)$, i.e., that the smooth monomorphism is necessarily
  a diffeomorphism with its image.
  The purpose of this paper is to prove this fact.   
  Thus, a Lie subgroup of $\Diff(\M)$ really is the same thing
  as the image in $\Diff(\M)$ of a smooth Lie group action on~$\M$.
  We give a formal statement in diffeological terms:
  
  {\sc Theorem} --- {\it Let $\M$ be a manifold.  Equip $\Diff(\M)$ 
  with the functional diffeology.  
  Then every smooth monomorphism from a Lie group to $\Diff(\M)$ 
  is an induction.}
  
  We now rephrase this statement in more common terms.
  Let $\rho \colon \G \to \Diff(\M)$ be an effective smooth action
  of a Lie group $\G$ on a manifold $\M$.  That is, $\rho$ is
  a one-to-one group homomorphism, 
  and the map $(g,m) \mapsto \rho(g)(m)$
  from $\G \times \M$ to $\M$ is smooth.   
  Let $D$ be an open subset of $\RR^k$ for some $k$.
  Let $r \mapsto f_r$ be a map from $D$ to the image of $\G$ in $\Diff(\M)$. 
  Let $r \mapsto g_r$ be the map from $D$ to $G$
  such that $\rho(g_r)=f_r$ for all $r \in D$.
  Assume that the map $(r,m) \mapsto f_r(m)$, from $D \times \M$ to $\M$,
  is smooth.  Then the map $r \mapsto g_r$, from $D$ to $G$, is smooth.  
  
  For example, this theorem immediately implies that
  the ordinary smooth structure of the linear group $\GL(n,\RR)$ is induced 
  by the functional diffeology of $\Diff(\RR^n)$.  This fact was proved 
  in \cite{Igl85} but not as a consequence of a more general theorem. 
  More generally, most of the classical Lie groups -- the orthogonal groups,
  the unitary groups, the symplectic groups etc.\ --
  are Lie subgroups (in the diffeological sense) of groups of diffeomorphisms 
  of vector spaces.
  
  Sentences such as
  ``What are, modulo conjugacy, the maximal subtori of a group of
  symplectomorphisms?'', as considered in \cite{KKP},
  have an intrinsic diffeological meaning. 
  The above theorem shows that, fortunately, this meaning coincides with 
  the usual interpretation of such sentences.
  
  In Section~\ref{sec:Smooth-Lie-group-actions-on-manifolds}, 
  we prove the above theorem.
  In Section~\ref{sec:Simple-cases-of-embeddings}, we give a simple 
  sufficient condition for which the induction is an embedding; this 
  condition holds for the classical Lie groups that are listed above.
  
  Throughout this paper, all manifolds, and in particular 
  all Lie groups, are assumed to be Hausdorff, second countable,
  and finite dimensional.
  
  
  \section{The diffeological framework}
  \labell{sec:The-diffeological-framework}
  
  We review here definitions and constructions in diffeology
  that we will need.  More details can be found in \cite{PIZ05-10}.
  
  Let $\X$ be a set.  A {\em parametrization} on $\X$ is a map 
  from an open subset of $\RR^n$, for some $n \in \NN$, to $\X$.
  A {\em diffeology} on $\X$ is a set of parametrizations,
  whose elements are called \emph{plots}, that satisfies 
  three axioms:  the {\em covering axiom}  -- 
  constant maps are plots;
  the {\em locality axiom} -- being a plot is a local condition:
  given $f \colon U \to \X$, if every point in $U$ has a neighbourhood
  $V$ such that $f|_V$ is a plot, then $f$ is a plot;
  and the {\em smooth compatibility axiom} -- 
  precomposition of a plot with a smooth map (in the usual sense) is a plot.
  A {\em diffeological space} is a set equipped with a diffeology.
  
  Let $\M$ be a manifold.  There exists on $\M$ a natural 
  diffeology: the plots consist of those parametrizations that are smooth 
  in the usual sense. 
  On the group $\Diff(\M)$ of diffeomorphisms 
  of $\M$ there also exists a natural diffeology, called the 
  \emph{functional diffeology}:  the plots consist of those
  parametrizations $U \to \Diff(\M)$, $r \mapsto f_r$,
  for which the map $(r,m) \mapsto f_r(m)$ from $U \times \M$ to $\M$
  is smooth.
  
  The \emph{discrete diffeology} on a set $\X$ is the diffeology
  for which only the locally constant parametrizations are plots.
  A space with the discrete diffeology is called \emph{discrete}.
  
  A {\em smooth map} between two diffeological spaces is a map
  whose precomposition with every plot of the source is a plot of the target.
  Diffeological spaces and their smooth maps form a category.
  The isomorphisms of this category are called {\em diffeomorphisms}. 
  Thus, a diffeomorphism is a bijective map that is smooth 
  and whose inverse is smooth.
  
  On any subset $\A$ of a diffeological space $\X$,
  there exists a natural diffeology, called the
  {\em subset diffeology}. Its plots are the plots of $\X$ with values
  in $\A$.  If this diffeology is discrete, $\A$
  is called a {\em discrete subspace}.   
  In particular, any countable subset of a manifold inherits the discrete
  diffeology, for example $\QQ \subset \RR$.

  An \emph{induction} is
  an injective smooth map between diffeological spaces
  that is a diffeomorphism onto its image, 
  where the image is equipped
  with the subset diffeology.  For manifolds, every induction is an
  immersion, and every immersion is locally an induction, but an
  immersion need not be an induction, even if it is injective.
  For example, the
  parametrization $t \mapsto (\sin(t), \sin(2t))$
  from $(-\pi,+\pi)$ to $\RR^2$, of the lemniscate (figure eight), 
  is an injective immersion
  but it is not an induction.

  A diffeology on $\X$ determines a topology on $\X$, 
  called the \emph{D-topology},
  in which a subset of $\X$ is open if and only if
  its preimage by every plot is open.  
  Such a subset is called \emph{D-open}.   
  If a subset of $\X$ is D-open, then its subset topology 
  coincides with the D-topology of its subset diffeology.
  The quotient $\RR/(\ZZ+\alpha\ZZ)$, for $\alpha$ irrational,
  is an example of a highly nontrivial\footnote{
  $\RR/(\ZZ+\beta\ZZ)$ is diffeomorphic to $\RR/(\ZZ+\alpha\ZZ)$ 
  if and only if there exists 
  $\left( \begin{smallmatrix} a & b \\ c & d \end{smallmatrix} \right) 
    \in \GL(2,\ZZ)$
  such that $\beta = \frac{a \alpha + b}{c \alpha + d}$
  \cite{DonIgl85}.}
  diffeological space whose D-topology is trivial.

  On a manifold, the D-topology coincides with the classical topology.
  A {\em manifold} can be redefined as a diffeological space in which
  every point is contained in a D-open set that is diffeomorphic
  to an open subset of some $\RR^n$.  A map of manifolds
  is smooth in the usual sense if and only if it is smooth in the
  diffeological sense. A {\em submanifold} is just a subspace which is
  a manifold.
  
  The D-topology distinguishes between inductions:
  an induction that is also a homeomorphism with its image,
  when the image is taken with the subset topology induced
  from the D-topology of the ambient space, is called an {\em embedding}. 
  Not every induction is an embedding.  For example, the 
  irrational line $\RR \to \RR^2/\ZZ^2$ given by $x \mapsto [x,\alpha x]$
  for $\alpha$ irrational is an induction but is not an embedding.
  
  The quotient of a diffeological space $\X$ by any equivalence relation 
  $\cR$ has a natural diffeology, called the {\em quotient diffeology}.
  Its plots are those parametrizations that can be locally 
  lifted to plots of $\X$ along the projection $\pi \colon \X \to \X/\cR$.
  Thus, $f \colon U \to \X/\cR$ is a plot if and only if
  for every point $u \in U$ there exists a neighbourhood $V$ of $u$ in $U$
  and a plot $\tilde{f} \colon V \to \X$ 
  such that $f|_{V} = \pi \circ \tilde{f}$.
  A {\em strict map}
  $f \colon \X \to \Y$ is a smooth 
  map that identifies the quotient $\X/f$ with the image $f(\X)$,
  as diffeological spaces \cite{Don84}.
  
  A \emph{diffeological group $\G$} is a group equipped with a diffeology 
  such that the inversion and the multiplication are smooth. 
  If $\M$ is a manifold then $\Diff(\M)$, with its functional
  diffeology, is a diffeological group.
  A Lie group is the same thing as a diffeological group
  that is also a manifold.
  A \emph{Lie subgroup} of $\Diff(\M)$ is a subgroup that, 
  with the subset diffeology, is a Lie group.
  
  A {\em smooth action} of a diffeological group
  $\G$ on manifold $\M$ is a smooth homomorphism 
  $\rho \colon \G \to \Diff(\M)$, where $\Diff(\M)$ is equipped 
  with the functional diffeology.  When $\G$ is a Lie group,
  this is equivalent to the usual definition of a smooth $\G$ action on $\M$:
  the map $(g,m) \mapsto \rho(g)(m)$ is required to be smooth.
  The action is \emph{effective} if the homomorphism $\rho$ is injective.
  The tautological action of $\Diff(\M)$ on $\M$ is smooth and effective.
  
  
  \section{Smooth Lie group actions on manifolds} 
  \labell{sec:Smooth-Lie-group-actions-on-manifolds}
  
  \begin{text}
    In the following paragraphs, we use the following notation.
    $\G$ is a Lie group, $\M$ is a manifold, 
    and $\rho \colon \G \to \Diff(\M)$ is a smooth action 
    of $\G$ on $\M$.  
    We denote the stabilizer of $m$ in $\G$ by $\G_m$.
    We denote $\rho(g)(m)$ by $g \cdot m$,
    and we denote the $\G$-orbit of $m$ by $\G \cdot m$.
    For every point $m$ of $\M$, we denote by $\orb{m} \colon \G \to \M$ 
    the \emph{orbit map} $\orb{m}(g) := g \cdot m$. 
    
    The orbit map $\orb{m}$ descends to a smooth map 
    $\orbbar{m} \colon \G/\G_m \to \M$. Thus, 
    $\orb{m} = \orbbar{m} \circ \pi$, where $\pi \colon \G \to \G/\G_m$ is
    the projection. Moreover, because 
    the stablizer $\G_m$ of $m$ is a closed subgroup of $\G$, 
    this stabilizer $\G_m$ is an embedded submanifold of $\G$, 
    the quotient $\G/\G_m$ is a manifold, and the projection 
    $\pi \colon \G \to \G/G_m$ is a submersion. Moreover, $\pi$ is a 
    (locally trivial)
    principal bundle with $\G_m$ as structure group acting on 
    $\G$ by $\G_m \times \G \ni (h,g) \mapsto gh^{-1}$; see for example 
    \cite{KobNom63}.
  \end{text} 

  \begin{claim}
    \artlabel{The orbit map} 
    \labell{The-orbit-map}
    \begin{it}
      Let $\M$ be a manifold, let $\G$ be a Lie group,
      and let $\rho \colon \G \to \Diff(\M)$ be a smooth action of $\G$ on $\M$.
      Let $m$ be a point of $\M$.
      Then the orbit map $\orb{m} \colon \G \to \M$ is a strict map.
    \end{it}
    
    {\sc Note 1.} This statement means that the map $\orbbar{m}$
    is a diffeomorphism from $\G/\G_m$ to $G \cdot m$, 
    where the former is equipped with the quotient diffeology 
    of $\G$
    and the latter is equipped with the subset diffeology induced from $\M$.
    In other words, $\orbbar{m}$ is an induction, which implies in particular
    that the orbit $\G \cdot m$ 
    is a submanifold. 
    
    {\sc Note 2.} The induction $\orbbar{m} \colon G/G_m \to \M$ 
    is not necessarily an embedding.    For example, consider the group
    $\G = \RR$, the manifold $\M = \RR^2/\ZZ^2 \times \RR$, and the action
    $\rho(t)([x,y],\alpha) = ( [x + t , y + \alpha t] , \alpha)$.
    Let $m = ([x,y],\alpha)$.
    If $\alpha$ is rational, then $\orb{m}$ is an embedding,
    and if $\alpha$ is irrational, then $\orb{m}$ is an induction
    but not an embedding.
  \end{claim} 
  
  \begin{proof} 
    Because the quotient map $\pi \colon \G \to \G/\G_m$ 
    admits smooth local sections, the map
    $\orbbar{m} \colon \G/\G_m \to \M$ is a smooth injection.
    Let $\sigma$ be a smooth section of $\pi$, let $\cO$ be its
    image in $\G$, and assume that $\id_\G \in \cO$. 
    Then $\cO$ is transverse to $\G_m$ at the identity $\id_\G$.
    
    The differential of the orbit map
    $\orb{m} \colon g \mapsto g \cdot m$, computed at the identity $\id_\G$, 
    is the map
    $\xi \mapsto \xi_\M(m)$, where $\xi$ is an element of the Lie algebra
    $\cG$ of $\G$ and $\xi_\M$ denotes the fundamental vector field on $\M$ 
    associated to $\xi$.
    The kernel of this map is equal to the Lie algebra $\cG_m$ 
    of $\G_m$:
    $$
      \cG_m = \T_{\id_\G} \G_m = \{ \xi \in \cG \mid \xi_\M(m) = 0 \}.
      $$
    The image of this map is the subspace
    $\cG \cdot m = \{ \xi_\M(m) \mid \xi \in \cG \}$
    of $T_m \M$.
    Let $\D$ be a disc in $\M$ through $m$ that is transverse 
    at $m$ to this subspace.

    Consider the map 
    $$ 
      \psi \colon \cO \times \D \to \M \ \ , \quad (a,x) \mapsto a \cdot x. 
    $$
    By the implicit function theorem,
        after possibly shrinking~$\cO$ and~$\D$, 
    the image $\U_m := \psi(\cO \times \D)$ is open in $\M$,
    and $\psi$ is a diffeomorphism of $\cO \times \D$ with $U_m$.

    For every $x \in D$, consider the set
    $$    \A_x := 
        \{ g \in \G \ | \ g \cdot m \in \psi(\cO \times \{ x \} ) \} . $$
    Note that $\A_x$ is nonempty if and only if $x \in \G \cdot m \cap \D$.

    Let us now show that
    \begin{equation} \tag{$\clubsuit$}
      \G \cdot m \cap \D \  \text{ is countable.}
    \end{equation}

    For this, we will first show that
    \begin{itemize}
    \item[($\diamondsuit$)]
      for every $x \in \D$, the set $A_x$ is an open subset of $\G$. 
    \item[($\heartsuit$)]
      for every $x,y \in \D$,
      if $x \neq y$ then $\A_x \cap \A_y = \varnothing$.
    \end{itemize}

    Let $g \in \G$. Suppose that $g \cdot m = a \cdot x$ for $a \in \cO$.
    The subset $\cO \cdot \G_x \cdot a^{-1} \cdot g$ of $\G$
    is a neighbourhood of $g$ in $\G$,
    and for every $g'$ in this subset, 
    $g' \cdot m \in \psi (\cO \times \{ x \} )$. 
    This shows ($\diamondsuit$).

    Now suppose that $g$ is in both $A_x$ and $A_y$.
    Then both $x$ and $y$ are equal to $\pr_\D \circ \psi^{-1}(g\cdot m)$,
    where $\psi^{-1} \colon \U_m \to \cO \times \D$
    is the inverse of $\psi$
    and where $\pr_\D \colon \cO \times \D \to \D$ is the projection
    to the second coordinate.
    So $x$ and $y$ are equal to each other.
    This shows ($\heartsuit$).

    Thus, $\{\A_x\}_{x \in \D}$ is a family of disjoint open subsets of $\G$.
    Because $\G$ is second countable, at most countably many elements 
    of this family are nonempty.  But $A_x$ is nonempty exactly if
    $x \in \G \cdot m \cap D$.  This shows~($\clubsuit$).

    Let us now check that the map $\orbbar{m} \colon \G/\G_m \to \M$ 
    is an induction (diffeomorphism with its image).
    Fix an arbitrary parametrization $\ol{\gamma} \colon \V \to \G/\G_m$,
    and assume that the composition $\orbbar{m} \circ \ol{\gamma}$ is smooth.
    We need to show that $\ol{\gamma}$ itself is smooth.

    Let $r_0$ be a point in $V$ and let $g \in G$ be such that
    $\ol{\gamma}(r_0)$ is equal to the coset $gG_m$ in $G/G_m$.
    We need to find a neighbourhood of $r_0$ in $V$ in which $\ol{\gamma}$
    is smooth.   Shrinking $V$, precomposing $\ol{\gamma}$ by 
    translation by $r_0$,
    and composing $\ol{\gamma}$ by left multiplication by the element $g^{-1}$
    of $G$, we reduce the problem to the following special case. 

    We assume that $\V$ is a small open ball centred at the origin~$0$, 
    that $\ol{\gamma}(0) = \pi(\id_\G)$,
    and that $\ol{\gamma}(\V) \subset \pi(\cO)$.
    The map $\gamma := \sigma \circ \ol{\gamma} \colon \V \to \cO$ 
    is then well defined, and it is a lift of $\ol{\gamma}$, that is,
    $\pi \circ \gamma = \ol{\gamma}$. 
    The map $v \mapsto \gamma(v) \cdot m$, being equal to the composition
    $\orbbar{m} \circ \ol{\gamma}$, is smooth by assumption; 
    also, it takes values in the subset $\U_m (= \psi(\cO \times \D))$ of~$\M$.

    By the definition of $\psi$, we can write
    $$\psi^{-1}(\gamma(v)\cdot m) = (\alpha(v),\mu(v)),$$ 
    where $\alpha$ 
    is a parametrization of $\G$ that takes values in $\cO$, 
    where $\mu$ is a parametrization of $\M$ that takes values in $\D$, 
    and where 
    $$\alpha(v) \cdot \mu(v) = \gamma(v) \cdot m$$
    for all $v \in \V$.  Rewriting the last equality as
    $\mu(v) = (\alpha(v)^{-1}\gamma(v)) \cdot m$, 
    we see that $\mu$ also takes values in $\G \cdot m$.
    So $\mu$ takes values in the set $\G \cdot m \cap \D$,
    which is countable by $(\clubsuit)$.

    Because $v \mapsto \psi^{-1}(\gamma(v) \cdot m)$, being the composition
    of two smooth maps, is smooth,
    the parametrizations $\alpha$ and $\mu$ are also smooth.
    Being a smooth map on an open ball that takes values in a
    countable set, the map $\mu$ is constant. So $\mu(v) = m$ for all $v$,
    and $\alpha(v) \cdot m = \gamma(v) \cdot m$ for all $v$.
    This, in turns, implies that $\pi(\alpha(v)) = \pi(\gamma(v))$ for all $v$.
    Because $\alpha(v)$ and $\gamma(v)$ both belong to $\cO$,
    composing with $\sigma$ gives $\gamma(v) = \alpha(v)$. 
    But we already know that $\alpha$ is smooth.  So $\gamma$ is also smooth,
    and, hence, so is $\ol{\gamma}$.
  \end{proof}

  \begin{claim}
    \artlabel{When $\G$ is countable} 
    \labell{When-G-is-countable}
    {\it Let $\M$ be a manifold, $\G$ a countable discrete group,
    and $\rho \colon \G \to \Diff(\M)$ a smooth action.
    Then the image of $\rho$ is a discrete subset of $\Diff(\M)$.
    In particular, if the action is effective, then $\rho$ is an induction.}
    
    {\sc Note 1.}  Because we assume manifolds to be second countable,
    a countable discrete group
    is the same thing as a zero dimensional Lie group.
    
    {\sc Note 2.} The image of a discrete diffeological group into 
    $\Diff(\M)$ by a smooth monomorphism has no reason to be discrete 
    {\em a priori}. For example, a Lie group $\G$ injects 
    into $\Diff(\G)$ 
    by left multiplication.  If $\G$ is equipped with the
    discrete diffeology, then the injection $\G \to \Diff(\G)$
    is smooth, but it is not an induction if $\dim G > 0$.
    We rule out this example by assuming that $\G$ is countable.
  \end{claim} 

  \begin{proof} 
    Let $r \mapsto f_r$ be a smooth parametrization 
    in $\Diff(\M)$ with values in $\rho(G)$, 
    and let $r_0$ be a point in the domain of this parametrization.
    Fix $m \in \M$.  Because $\G$ is countable,
    the orbit $G \cdot m$ is countable.
    But a countable subspace of a manifold is diffeologically discrete.
    So the smooth parametrization $\varphi \colon r \mapsto f_r(m)$, 
    which takes its values 
    in the countable subspace $\G \cdot m$ of $M$, is locally constant. 
    But locally constant means constant on the connected components 
    of the domain of the parametrization \cite[chap.~II]{PIZ05-10}. 
    Let $\B$ be an open ball in the domain of $\varphi$ 
    that is centred at $r_0$.
    Then $f_r(m) = f_{r_0}(m)$ for all $r$ in $\B$.
    Repeating this argument for all $m$ in $\M$,
    we deduce that $f_r = f_{r_0}$ for all $r$ in $\B$. 
    Thus, the parametrization $r \mapsto f_r$ is locally constant. 
    In conclusion, the image of $\G$ by $\rho$ is discrete. 
    It follows immediately that if the action is effective
    then $\rho$ is an induction.
  \end{proof}
  
  \begin{claim}
    \artlabel{The case of a discrete stabilizer}
    \labell{The-case-of-a-discrete-stabilizer}
    {\it Let $\M$ be a manifold, let $\G$ be a Lie group,
    and let $\rho \colon \G \to \Diff(\M)$ be an effective smooth action 
    of $\G$ on $\M$. 
    Suppose that there exists a point $m$ in $\M$ such that the stabilizer
    $\G_m$ is discrete.  Then $\rho$ is an induction.}
  \end{claim} 
  
  \begin{proof}
    Let $r \mapsto f_r$ be a parametrization of $\Diff(\M)$ 
    that takes its values in $\rho(\G)$.
    Let $r \mapsto g_r$ be the parametrization of $\G$
    such that $\rho(g_r) = f_r$ for all $r$.
    Let $r_0$ be a point in the domain of these parametrizations.
    We would like to show that $r \mapsto g_r$ is smooth near $r_0$.
    
    Let $m$ be a point of $\M$ whose stabilizer $\G_m$ is discrete.
    The parametrization $r \mapsto f_r(m)$ of $\M$
    takes its values in the orbit $\G \cdot m$. Since 
    $\orbbar{m} \colon \G/\G_m \to \M$ is an induction
    (Paragraph \ref{The-orbit-map}), 
    and by the definition of the quotient diffeology on $\G/\G_m$,
    there exists a neighbourhood $B$ of $r_0$ and a smooth map
    $r \mapsto g_r'$ from $B$ to $\G$ 
    such that $f_r(m) = \rho(g_r')(m)$ for all $r \in B$. 
    The map $r \mapsto \rho(g_r')^{-1} \circ f_r$ is a plot of $\Diff(\M)$
    that takes its value in $\rho(\G_m)$. But since
    $\G_m$ is assumed to be discrete, $\rho(\G_m)$ is also discrete 
    (Paragraph~\ref{When-G-is-countable}). So, 
    $r \mapsto  \rho(g_r')^{-1} \circ f_r$ is locally constant.  
    Therefore, after possibly shrinking the neighbourhood $B$ of $r_0$,
    there exists a constant element $\kappa \in G_m$
    such that $\rho(g_r')^{-1} \circ f_r = \rho(\kappa)$, that is,
    $f_r = \rho(g_r' \kappa)$, for all $r \in B$.
    But also $f_r = \rho(g_r)$.  Because the action is effective,
    we deduce that $g_r = g_r' \kappa$ for all $r \in B$.
    Because $r \mapsto g_r'$ is smooth and $\kappa$ is constant,
    this implies that $r \mapsto g_r$ is smooth.
  \end{proof}
  
  \begin{claim}
    \artlabel{Intersecting vector subspaces} 
    \labell{Intersecting-vector-subspaces}
    {\it
    Let $\E$ be a finite dimensional vector space.
    Let $\{\E_i\}_{i \in \cI}$ be a family of subspaces
    of $\E$. If $\cap_{i \in \cI} \E_i = \{0\}$, then there exists
    a finite set of indices $(i,j,\ldots,k) \subset \cI$ such that
    $\E_{i} \cap \E_j \cap \cdots \cap \E_{k} = \{0\}$.}
  \end{claim} 
  
  \begin{proof}
    We prove the result by induction on the dimension of $E$.
    If $\E = \{ 0 \}$, the result is trivial: we can take
    the finite set of indices to be the empty set.
    Suppose that the result is true for vector spaces 
    of dimension smaller than that of $E$.
    Because $\cap_{i \in \cI} \E_i = \{0\}$, 
    there exists at least one vector subspace $\E_{i}$ such that
    $\E_{i} \neq \E$, that is $\dim(\E_{i}) < \dim(\E)$. 
    Let us define $\cI' := \cI \smallsetminus \{i\}$, 
    $\E' := \E_{i}$, and 
    $\E'_{j} := \E' \cap \E_{j} \text{ \ for all } j \in \cI'$.
    We have $\cap_{j \in \cI} \E_j = \cap_{j \in \cI'} \E_{j}$.
    Thus, $\{E_{j}\}_{j \in \cI'}$ is a family of
    vector subspaces of $\E' (=\E_{i})$ with 
    $\cap_{j \in \cI'} \E'_{j} = \{0\}$.
    By the induction hypothesis, there is a finite set
    of indices $\{ j , \ldots , k \} \subset \cI'$
    such that $\E'_{j} \cap \ldots \cap \E'_{k} = \{ 0 \}$.
    The union of this set with the index $\{ i \}$
    is a finite set of indices $\{ i , j , \ldots , k \} \subset \cI$
    such that $\E_i \cap \E_j \cap \ldots \cap \E_k = \{ 0 \}$.
  \end{proof}
  
  \begin{claim}
    \artlabel{Obtaining a discrete intersection of stabilizers}
    \labell{Obtaining-a-discrete-stabilizer} 
    {\it Let $\M$ be a manifold, $\G$ a Lie group, 
    and $\rho \colon \G \to \Diff(\M)$ an effective smooth action.
    Then there exists a finite set of points
    $\{ m_1,\ldots,m_\N \} \subset \M$ such that the intersection
    of their stabilizers, $\G_{m_1} \cap \ldots \cap G_{m_\N}$,
    is discrete.}
  \end{claim} 
  
  \begin{proof}
    Denote by $\cG$ the Lie algebra of $\G$,
    and, for every point $m$ of $\M$, 
    denote by $\cG_m$ the Lie algebra of the stabilizer $\G_m$ of $m$.
    For any smooth parametrization $\P \colon \U \to \G$,
    if $\P$ takes values in $G_m$, then, for every $r \in \U$,
    the composition of the differential 
    $d\P_r \colon T_r \U \to T_{\P(r)} \cG$
    with the left translation by $P(r)^{-1}$ takes values in 
    the subspace $\cG_m$ of $\cG$.
    
    Because the action is effective, 
    $\cap_{m \in \M} \G_m = \{ 1 \}$.
    Recall that $\cG_m = \{ \xi \in \cG \ | \ 
      \exp(t\xi) \in \G_m \text{ for all } t \in \RR \}$.
    It follow that
    $\cap_{m \in \M} \cG_m = \{ 0 \}$.
    According to Paragraph~\ref{Intersecting-vector-subspaces}, there
    exists a finite set of points, say, $\{ m_1, \ldots, m_\N \} \subset \M$,
    such that $\cG_{m_1} \cap \cdots \cap \cG_{m_\N} = \{0\}$.  
    
    Let $\P \colon \U \to \G$ be a smooth parametrization that takes values
    in $\G_{m_1} \cap \cdots \cap \G_{m_\N}$.
    For every $r \in \U$, the composition of the differential 
    $d\P_r \colon T_r \U \to T_{\P(r)} \G$
    with the left translation by $P(r)^{-1}$ takes values 
    in $\cG_{m_1} \cap \ldots \cap \cG_{m_\N}$.
    So the differential $d P_r$ is zero for every $r \in \U$,
    and so the map $\P$ is locally constant.
    Therefore, $\G_{m_1} \cap \cdots \cap \G_{m_\N}$ is discrete.
  \end{proof}
  
  \begin{claim}
    \artlabel{Smooth Lie group actions}
    \labell{smooth-Lie-group-actions}
    {\it
    Let $\M$ be a manifold, $\G$ a Lie group,
    and $\rho \colon \G \to \Diff(\M)$ an effective smooth action.
    Then $\rho$ is an induction. That is, $\rho$ is a diffeomorphism 
    onto its image, when the image is equipped with the subset diffeology.}
    
    {\sc Note.}   Let $\M$ be a manifold, $\G$ a Lie group,
    and $\rho \colon \G \to \Diff(\M)$ a smooth action
    that is not necessarily effective.  Let $\K = \ker \rho$.
    Applying Claim~\ref{smooth-Lie-group-actions} to the quotient $\G/\K$,
    we deduce that $\rho$ induces a diffeomorphism 
    from $\G/\K$ to its image in $\Diff(\M)$.
  \end{claim} 
  
  \begin{proof} 
    By Paragraph~\ref{Obtaining-a-discrete-stabilizer}, there exists 
    a natural number $\N$ and a point $\mu = (m_1, \ldots, m_\N)$
    in $\M^\N$ whose stabilizer $\G_\mu$ under the diagonal action
    of $\G$ on $\M^\N$ is discrete.   Write this diagonal action as
    $\rho_\N \colon \G \to \Diff(\M^\N)$, that is,
    $\rho_\N(g)(m_1,\ldots,m_\N) = (g \cdot m_1 , \ldots , g \cdot m_\N)$.
    By Paragraph~\ref{The-case-of-a-discrete-stabilizer}, 
    the map $\rho_N \colon \G \to \Diff(\M^\N)$ is an induction.
    Now, let $r \mapsto f_r$ be a plot of $\Diff(\M)$ 
    that takes its values in $\rho(\G)$, and, for every $r$,
    let $g_r$ be the element of $\G$ such that $f_r = \rho(g_r)$.
    Let $f_{\N,r} (m_1,\ldots,m_\N) := (f_r(m_1),\ldots, f_r(m_\N))$.
    Then $r \mapsto f_{\N,r}$ is a plot of $\Diff(\M^\N)$
    that takes its values in $\rho_\N(\G)$,
    and $f_{\N,r} = \rho_\N(g_r)$ for all $r$.
    Because $\rho_\N$ is an induction,  it follows that
    $r \mapsto g_r$ is smooth.  Thus, $\rho$ is an induction.
  \end{proof}
  

  \section{A simple case of embedding} 
  \labell{sec:Simple-cases-of-embeddings}

  \begin{text}
    We have proved the theorem that we stated in the introduction:
    an effective smooth Lie group action $\rho \colon \G \to \Diff(\M)$
    is an induction.  

    It is now natural to ask when is such $\rho$ an \emph{embedding},
    that is, a homeomorphism of $\G$ with its image, when the image
    is equipped with the subset topology of the D-topology of $\Diff(\M)$. 
    In other words, how do these two topologies compare?
    Paragraph~\ref{An-orbit-map-embedding-condition} gives a 
    sufficient condition. Applied to 
    $\GL(n,\RR)$ acting naturally on $\RR^n$, this condition shows
    again that $\GL(n,\RR)$ is embedded in $\Diff(\RR^n)$ \cite{Igl85}.
  \end{text}

  \begin{claim}
    \artlabel{An orbit map embedding condition}
    \labell{An-orbit-map-embedding-condition}
    {\it Let $\M$ be a manifold, $\G$ a Lie group,  
    and $\rho \colon \G \to \Diff(\M)$ an effective smooth action.
    Suppose that there exists a natural number $k$ and a point
    $\mu = (m_1, \ldots, m_k)$ in $\M^k$ 
    such that the orbit map 
    $\orb{\mu} \colon \G \to \M^k$ for the diagonal action is an embedding.
    Then $\rho \colon \G \to \Diff(\M)$ is an embedding.}

    {\sc Note 1.} Considering the $n$ standard basis vectors of $\RR^n$,
    we conclude that $\GL(n,\RR)$ is embedded
    in $\Diff(\RR^n)$. The same applies to classical Lie groups
    such as the orthogonal groups, the unitary groups, the symplectic groups,
    etc.  

    {\sc Note 2.}  Every diffeological group is embedded (through 
    left translation) 
    in its group of diffeomorphisms \cite[chap.~VII]{PIZ05-10}.
    In particular, $\T^2$ is embedded in $\Diff(\T^2)$. 
    The irrational flow $\R \to \Diff(\T^2)$, given by
    $\rho(t)([x,y]) = [x+t,y+\alpha t]$ with $\alpha$ irrational, 
    splits as a composition $\R \to \T^2 \to \Diff(\T^2)$.
    Because the first arrow is not an embedding 
    (cf.~Note~2 of Paragraph~\ref{The-orbit-map})
    and the second arrow is an embedding,
    {\em a fortiori}, this flow is not embedded in $\Diff(\T^2)$.
    Nevertheless, by our main theorem, this flow \emph{is} 
    a diffeomorphism with its image in $\Diff(\T^2)$.

  \end{claim} 

  \begin{proof} 
    Let $\B \subset \G$ be an open set, and let 
    $$
      \U = \orb{\mu}(\B) = \{(g \cdot m_1, \ldots, g \cdot m_k) \in \M^k 
      \mid g \in \B \}.
      $$
    Since, by assumption, $\orb{\mu}$ is an embedding, $\U$ is open 
    in $\G \cdot \mu$ with respect to the subset topology. So, 
    there exists an open subset $\V \subset \M^k$
    such that $\U = \V \cap \G \cdot \mu$.  Let us introduce the set
    $$
      \Omega = \{ f \in \Diff(\M) \mid (f(m_1),\ldots,f(m_k)) \in \V \}.
      $$
    Let us check now that $\Omega$ is open for the D-topology of the
    functional diffeology
    and that $\rho(B) = \rho(\G) \cap \Omega$. This will prove the claim.
    
    Let $\P \colon \cO \to \Diff(\M)$ be a plot.
    The preimage $\P^{-1}(\Omega)$ consists of those $r \in \cO$ 
    such that $$(\P(r)(m_1),\ldots,\P(r)(m_k)) \in \V.$$
    That is, $\P^{-1}(\Omega) = (\orb{\mu} \circ \P)^{-1} (\V)$. But
    $\orb{\mu} \circ \P$ is smooth, hence D-continuous, and $\V$ is open. 
    Thus, $\P^{-1}(\Omega)$ is open. 
    Because this is true for every plot $\P \colon \cO \to \Diff(\M)$,
    we deduce that $\Omega$ is D-open.
    
    Let us check now that $\rho(B) = \rho(\G) \cap \Omega$.
    Let $g' \in \G$. By the definition of $\Omega$, 
    the diffeomorphism $\rho(g')$ belongs to $\Omega$ if and only if 
    the $k$-tuple $(g' \cdot m_1 , \ldots , g' \cdot m_k)$ is in $\V$.   
    But this $k$-tuple is automatically in $\G \cdot \mu$. 
    So $\rho(g')$ belongs to $\Omega$ if and only if
    $(g' \cdot m_1 , \ldots , g' \cdot m_k)$ belongs to 
    the intersection $\V \cap \G \cdot \mu$, which, in turn, is equal to $\U$,
    which is $\orb{\mu}(\B)$. 
    Thus, $\rho(g')$ belongs to $\Omega$ if and only if
    there exists $g \in \B$ such that 
    $(g' \cdot m_1, \ldots, g' \cdot m_k)
      = (g \cdot m_1, \ldots, g \cdot m_k)$.   This last condition holds
    exactly if $g^{-1} g' \in \G_\mu$. 
    Because $\G_\mu$ is trivial, $\rho(g')$ belongs to $\Omega$
    if and only if $g' = g \in \B$.  
    
    Therefore, $\Omega$ is D-open in $\Diff(\M)$, 
    and $\rho(B) = \rho(\G) \cap \Omega$.
    So $\rho(B)$ is open with respect to the subset topology on $\rho(\G)$ 
    that is induced from the D-topology in $\Diff(\M)$. 
    The proof is complete. 
  \end{proof}
  
  \newpage


\end{document}